\documentclass[12pt]{extarticle}
\usepackage{amsmath, amsthm, amssymb, color}
\usepackage{graphicx}
\usepackage{caption}
\usepackage{mathtools}
\usepackage{enumitem}
\usepackage{verbatim}
\usepackage{longtable}
\usepackage{pifont}
\usepackage{makecell}
\usepackage{tikz}
\usetikzlibrary{matrix}
\usetikzlibrary{arrows}
\usepackage{algorithm}
\usepackage[noend]{algpseudocode}
\usepackage{caption}
\usepackage[normalem]{ulem}
\usepackage{subcaption}
\usepackage{algorithm}
\usepackage{algpseudocode}
\usepackage{xcolor}

\usepackage[colorlinks,plainpages,hypertexnames=false,plainpages=false]{hyperref}
\hypersetup{urlcolor=blue, citecolor=blue, linkcolor=blue}

\oddsidemargin \evensidemargin
\usepackage{makecell}
\usepackage{multirow,array}

\newtheorem{theorem}{Theorem}
\newtheorem{proposition}[theorem]{Proposition}
\newtheorem{lemma}[theorem]{Lemma}
\newtheorem{corollary}[theorem]{Corollary}

\theoremstyle{definition}
\newtheorem{algo}[theorem]{Algorithm}

\newtheorem{remark}[theorem]{Remark}

\newtheorem{example}[theorem]{Example}

\numberwithin{theorem}{section}

\newcommand{\PP}{\mathbb{P}}
\newcommand{\RR}{\mathbb{R}}
\newcommand{\QQ}{\mathbb{Q}}
\newcommand{\CC}{\mathbb{C} }

\newcommand{\xmark}{\ding{55}}

\newcommand{\real}{\mathrm{real}}
\newcommand{\Aut}{\mathrm{Aut}}

\newcommand{\Lazypic}[2]{\begin{minipage}{#1} \vspace{0.1cm} \centering {#2}\vspace{0.1cm}\end{minipage}}

\title{\bf Recovery of Plane Curves from Branch Points}

\author{Daniele Agostini, Hannah Markwig, Clemens Nollau, \\
             Victoria Schleis, Javier Sendra--Arranz, and Bernd Sturmfels}

 \date{ }

\begin{document}
\maketitle
 \begin{abstract}
\noindent We recover
  plane curves from their branch points
 under projection onto a 
 line. Our focus lies on
 cubics and quartics. These have $6$ and $12$ branch points respectively.
The plane Hurwitz numbers $40$ and $120$ count the orbits of solutions.
 We determine the numbers of real solutions, and we present
 exact algorithms for recovery. Our approach relies on $150$ years
of  beautiful algebraic geometry, from Clebsch to Vakil and beyond.
 \end{abstract}

\section{Introduction}
\label{sec1}

Arthur Cayley in 1879  was the first to use  ``algorithm'' to title a
discrete geometry paper. In \cite{CayleyAlgo} he identifies
 the finite vector space $(\mathbb{F}^2)^6$ 
with  the $64$ theta characteristics 
of a plane quartic curve, i.e.~the $28$ bitangents and the $36$ symmetric determinantal representations.
The present paper can be viewed as a sequel. Our Table \ref{table:40covers}
is very much in the spirit of \cite{CayleyAlgo}. 

One century after Cayley, algorithms in discrete geometry  became a field
in its own~right, in large part thanks to Eli Goodman. 
We are proud to dedicate this article to Eli's memory.
Eli obtained his PhD in 1967 with Heisuke Hironaka. He had important 
publications in algebraic geometry (e.g.~\cite{Goodman})
 before embarking
on his distinguished career on the discrete~side.

\smallskip

Consider the map
$\pi: \PP^2 \dashrightarrow \PP^1$
that takes a point $(x:y:z) $ in the projective plane
to the point $(x:y)$ on the projective line. Geometrically,
this is the projection with center
 $p = (0:0:1)$. We restrict $\pi$ to the
 curve $V(A)$ defined by a general 
 ternary form of degree~$d$,
\begin{equation}
\label{eq:intro_f} A(x,y,z) \,\,\,=\, \sum_{i+j+k=d} \! \alpha_{ijk} \,x^i y^j z^k .
\end{equation}
The resulting $d:1$ cover $V(A) \rightarrow \PP^1$
 has $d(d-1)$ branch points, represented by a 
  binary~form
 \begin{equation}
 \label{eq:intro_g} B(x,y) \,\,\, =  \sum_{i+j=d(d-1)} \!\! \beta_{ij}\, x^i y^j. 
 \end{equation}
Passing from the curve to its branch points 
  defines a rational map 
 from the space $ \PP^{\binom{d+2}{2}-1}$ with coordinates
 $\alpha$ to the space $\PP^{d(d-1)} $
 with coordinates $\beta$.
Algebraically, this is the map
\begin{equation}
\label{eq:map1}
 \PP^{\binom{d+2}{2}-1} \,\dashrightarrow\,\, \PP^{d(d-1)} \,,\, \,A \,\mapsto \, 
 {\rm discr}_z(A). 
\end{equation}
This is the discriminant of $A$ with respect to the last variable.
That discriminant is a binary form $B$ of degree $d(d-1)$ in $x,y$ whose coefficients are polynomials
of degree $2d-2$ in $\alpha $.

We here study the {\em Inverse Problem}, 
namely recovery of the curve from its branch points.
Given the binary form $B$, our task is to
compute all ternary forms $\hat A$ such that $ {\rm discr}_z(\hat A) = B$.
This is a system of  $d(d-1)+1$ polynomial
equations of degree $2d-2$ in the $\binom{d+2}{2}$ unknowns $\alpha$.
Solving this system means computing a fiber of the map (\ref{eq:map1}) over $B$.
Recovery is not unique because ${\rm discr}_z(A)$ 
is invariant under the action of the subgroup $\mathcal{G}$ of ${\rm PGL}(3)$ 
given~by
\begin{equation}
\label{eq:groupG} \qquad
g \,\,:\,\, x \mapsto g_0 x\,,
\,\,y \mapsto g_0 y \, , 
\,\, z \mapsto g_1 x + g_2 y + g_3 z
\qquad \hbox{with $\,g_0 g_3 \not=0$.}
\end{equation}
By \cite[Proposition 5.2.1 and Corollary 5.2.1]{Ongaro}, the fiber over $B$
is a finite union  of $\mathcal{G}$-orbits.
Their number 
$\mathfrak{h}_d$ is the {\em plane Hurwitz number} of degree $d$.
Our task is to compute representatives for all
$\mathfrak{h}_d$ orbits in the fiber of the map (\ref{eq:map1}) over 
a given binary form $B$.

\begin{example}[$d=2$] For conics we have $\mathfrak{h}_2 = 1 $
and recovery is easy. Our polynomials are
$$
\begin{matrix}
A & = &  \alpha_{200}  x^2 +
     \alpha_{110}  x y +
\alpha_{101} x z +
\alpha_{020} y^2 +
\alpha_{011} y z +
\alpha_{002 } z^2, \\
{\rm discr}_z(A) & = &
(4 \alpha_{002} \alpha_{200}-\alpha_{101}^2) x^2
\,+\,(4\alpha_{002} \alpha_{110}-2 \alpha_{011} \alpha_{101}) x y
\,+\,(4 \alpha_{002} \alpha_{020}-\alpha_{011}^2) y^2, \\
B & = &  \beta_{20} x^2 + \beta_{11} xy + \beta_{02} y^2.
\end{matrix}
$$
The equations ${\rm discr}_z(\hat A) = B$ describe
precisely one
$\mathcal{G}$-orbit  in $\PP^5$. A point in that orbit is
$$  \hat A \,\,= \,\, \frac{1}{4}\beta_{20} x^2 +  \frac{1}{4} \beta_{11} x y - 
\beta_{02} y z + \beta_{02} z^2. $$
Up to the $\mathcal{G}$-action, this is the unique
 solution to our recovery problem for plane conics.
\hfill $ \diamond$ \end{example}

Plane Hurwitz numbers $\mathfrak{h}_d$ were studied in Ongaro's 2014 PhD~thesis and in his work with Shapiro \cite{Ongaro, OS}.
These served as the inspiration for our project. Presently, the only
known nontrivial values are $\mathfrak{h}_3 = 40$ and
$\mathfrak{h}_4 = 120$. The former value is
due to Clebsch~\cite{ClebschShort, ClebschLong}.
We first learned it from
\cite[Proposition 5.2.2]{Ongaro}.
The latter value was computed by Vakil in \cite{Ravi}. 
The plane Hurwitz number $\mathfrak{h}_4 =120$ was presented with the extra factor $(3^{10}-1)/2$ in 
\cite[eqn.~(5.14)]{Ongaro} and  in \cite[p.~608]{OS}. However,
 that factor is not needed; see
Remark~\ref{rmk:extrafactor}.

The parameter count above implies that the closure of the image of (\ref{eq:map1}) 
is a variety $\mathcal{V}_d$ of dimension $\binom{d+2}{2}-4$ in an
ambient space of dimension $d(d-1)$. For $d=2,3$, the two
dimensions agree, so recovery is possible for generic $B$.
For $d \geq 4$, the
constraint $B \in \mathcal{V}_d$ is nontrivial.
For instance,  $\mathcal{V}_4$
is a hypersurface of degree $3762$ in $\PP^{12}$,
as shown by Vakil \cite{Ravi}.

\smallskip

This article is organized as follows.
In Section \ref{sec2} we approach our problem
from the perspective of computer algebra.
We establish a normal form with respect to the
$\mathcal{G}$-action, and we identify the base locus of the
map (\ref{eq:map1}). This allows
to state the recovery problem as a polynomial system
with finitely many solutions over the  
 complex numbers $\CC$.
The number of solutions is $\mathfrak{h}_3 = 40$
for cubics, and  it is $\mathfrak{h}_4 = 120$,
provided $B$ lies on the hypersurface $\mathcal{V}_4$.

In Section~\ref{sec3} we establish the relationship to
Hurwitz numbers that count abstract coverings of $\PP^1$.
We encode such coverings by monodromy graphs, and we
determine the real Hurwitz numbers for our setting.
A highlight is Table \ref{table:40covers}, which matches
the $40$ monodromy representations for $d=3$ with combinatorial labels
taken from Clebsch \cite{ClebschLong} and Elkies  \cite{elkies}.

In Section~\ref{sec4} we exhibit the Galois group for the
$40$ solutions when $d=3$, and we discuss different
realizations of this group.
 Theorem  \ref{thm:25920} implies
that it agrees with the Galois group for the $27$ lines on the cubic surface.
Following classical work of Clebsch \cite{ClebschShort, ClebschLong},
we   show that the recovery of the $39$ other cubics 
from the given cubic $A$ can be solved in radicals.

Section~\ref{sec5} builds on work of Vakil \cite{Ravi}.
It relates the recovery of quartic curves to tritangents
of sextic space curves and to del Pezzo surfaces
of degree one. Theorem \ref{thm:realcount4planar}
determines the possible number of real solutions.
Instances with $120$ rational solutions can be constructed 
by blowing up the  plane $\PP^2$ at $8$ rational points.
We conclude with Theorem \ref{thm:rleqs} which connects the
real structure of $8$ points in $\PP^2$ with that of the $12$ branch points in $\PP^1$.

This article revolves around explicit computations, summarized in
Algorithms \ref{algo:recovery4}, \ref{alg:recovery3}, \ref{alg:clebsch}, \ref{alg:get8},  \ref{alg:get120}.
Our software and other supplementary material  is 
available at the repository website {\tt MathRepo}~\cite{mathrepo} of MPI-MiS via the link
\href{https://mathrepo.mis.mpg.de/BranchPoints/}{https://mathrepo.mis.mpg.de/BranchPoints}$\,$.

\section{Normal Forms and Polynomial Systems}
\label{sec2}

We identify $\PP^{\binom{d+2}{2}-1}$ 
with the space of plane curves (\ref{eq:intro_f}) of degree $d$ and use as homogeneous  coordinates the $\alpha_{ijk}$.
The following subspace of that projective space has codimension three:
\begin{equation}
\label{eq:Ld} L_d \,\, = \,\,V(\,\alpha_{1 0 \,d-1}\,,\,\alpha_{d-1 \, 1 0 }\,, \,
\alpha_{00d} - \alpha_{01 \,d-1} \,). 
\end{equation} 
We now show that
this linear space serves as normal form with respect to the group 
action on fibers of (\ref{eq:map1}).  The group that acts is the three-dimensional
group $\mathcal{G} \subset {\rm PGL}(3)$ given in~(\ref{eq:groupG}).

\begin{theorem} \label{thm:normalform}
Let $A$ be a ternary form of degree $d\geq 3$ such that
\begin{equation}
\label{eq:genericity}
	\displaystyle
	\alpha_{00d}\left(\, \sum_{k=0}^{d-1}\frac{(k+1)(-1)^k}{d^k}\alpha_{10\,d-1}^k\alpha_{00d}^{d-k-1}\alpha_{d-k-1\,0\,k+1}
	\right)\,\,\neq \,\,0.
\end{equation}
The  orbit of $\, A$  under the $\mathcal{G}$-action on
$\,\PP^{\binom{d+2}{2}-1}$ intersects the linear space $L_d$ in one point.
\end{theorem}

\begin{remark}
This statement is false for $d=2$.  The
$\mathcal{G}$-orbit of $A$ consists of the conics
\begin{align*}
& g A \,=\,
(\alpha_{002} g_1^2+\alpha_{101} g_0 g_1+\alpha_{200} g_0^2) x^2
+(2 \alpha_{002} g_1 g_2+\alpha_{011} g_0 g_1\,+\,\alpha_{101} g_0 g_2+\alpha_{110} g_0^2) x y \,\, +  \\&
(2 \alpha_{002} g_1 g_3{+}\alpha_{101} g_0 g_3) x z
+(\alpha_{002} g_2^2{+}\alpha_{011} g_0 g_2{+}\alpha_{020} g_0^2) y^2
+(2 \alpha_{002} g_2 g_3{+}\alpha_{011} g_0 g_3) y z
 \!+\!\alpha_{002} g_3^2 z^2.
\end{align*}
For generic $\alpha$,
no choice of $g \in \mathcal{G}$ makes
both the $xy$-coefficient and the $xz$-coefficient zero.
Note that the parenthesized sum in (\ref{eq:genericity}) is the zero polynomial for $d=2$,
but not for $d \geq 3$.
\end{remark}

\begin{proof}[Proof of Theorem~\ref{thm:normalform}]
The unique point in $\,L_d \,\cap \,\mathcal{G} A\,$
is found by computation.
Without loss of generality, we set $g_0=1$.
Next we set $g_1 = -\frac{1}{d} \alpha_{10 \,d-1}/ \alpha_{00d}$ because
the coefficient of $xz^{d-1}$ in $gA$ equals
$(d \alpha_{00d} g_1 + \alpha_{10 \,d-1}) g_3^{d-1}$.
The polynomial $gA$ arises from $A$ by the coordinate change $z \mapsto g_1x+g_2y+g_3z$. Thus, a monomial $x^iy^jz^{d-i-j}$ contributes the expression $x^iy^j(g_1x+g_2y+g_3z)^{d-i-j}$ to $gA$. This  contributes to the monomials  $x^{i'}y^{j'}z^{d-i'-j'}$ with $i'\geq i$ and $j'\geq j$. The coefficient of $x^{d-1}y$
in $gA$ arises from the following subsum of $A$:
$$\sum_{i=0}^{d-1} \alpha_{i0\,d-i}\,x^iz^{d-i}\,+\,\sum_{i=0}^{d-1} \alpha_{i1\,d-i-1}\,x^iyz^{d-i-1},$$
after inserting the coordinate change.
Thus the coefficient of $x^{d-1}y$ in $gA$ equals
$$\sum_{i=0}^{d-1}   \alpha_{i0\,d-i}(d-i)\,g_1^{d-i-1} g_2 \,+\,\sum_{i=0}^{d-1} \alpha_{i1\,d-i-1}\,g_1^{d-i-1}.$$
Inserting the above result for $g_1$, and setting the coefficient of $x^{d-1}y$ to zero, we can
solve this affine-linear equation for $g_2$, obtaining a rational function in the $\alpha_{ijk}$ as solution for $g_2$.

Next, we equate the
coefficients of $y z^{d-1} $ and $z^d$.
The first can be computed from the subsum $\,\alpha_{00d}z^d\,+\,\alpha_{01\,d-1}yz^{d-1}$ and equals
$\,\alpha_{00d}\, d\, g_2 g_3^{d-1}\,+\,\alpha_{01\,d-1}\, g_3^{d-1}$. The second 
is computed from the $z^d$ coefficient of $A$ only, and we find it to be $\alpha_{00d}\cdot g_3^d$.
Setting these two equal and solving for $g_3$, we obtain
$\,g_3= \frac{1}{\alpha_{00d}}\,(\alpha_{00d}\, d\, g_2+\alpha_{01\,d-1})$.
Inserting our result for $g_2$, we obtain a rational function in the $\alpha_{ijk}$ as solution for $g_3$.
\end{proof}

\begin{example}
To be explicit, we display the solution in the two cases of primary interest.
For cubics $(d=3)$, the unique point $gA$ in  $\,L_3 \,\cap \,\mathcal{G} A\,$ is
given by the group element $g$ with
$$
g_0 = 1 ,\,\,
g_1 \,=\, -\frac{\alpha_{102}}{3 \alpha_{003}},\,\,
g_2 \,=\, 
\frac{9 \alpha_{003}^2 \alpha_{210}-3 \alpha_{003} \alpha_{102} \alpha_{111}
+\alpha_{012} \alpha_{102}^2}{3\alpha_{003}(3 \alpha_{003} \alpha_{201}-
\alpha_{102}^2)},
$$
$$
g_3 \,\,=\,\, \frac{9 \alpha_{003}^3 \alpha_{210}+3 \alpha_{003} \alpha_{012} \alpha_{201}
-3 \alpha_{003}^2 \alpha_{102} \alpha_{111}+\alpha_{003} \alpha_{012} \alpha_{102}^2-\alpha_{102}^2\alpha_{012}}
{\alpha_{003} (3 \alpha_{003} \alpha_{201}-\alpha_{102}^2)}.
$$

For quartics $(d=4)$, the unique point $gA$ in  $\,L_4 \,\cap \,\mathcal{G} A\,$ is
given by $g \in \mathcal{G}$, where
$$
g_0 = 1,\,\,
g_1 \,=\, -\frac{\alpha_{103}}{4 \alpha_{004}},\,\,
g_2 \,=\, 
\frac{64 \alpha_{004}^3 \alpha_{310}-16 \alpha_{004}^2 \alpha_{103} \alpha_{211}
+4 \alpha_{004} \alpha_{103}^2 \alpha_{112}-\alpha_{013} \alpha_{103}^3)}
{8 \alpha_{004}(8 \alpha_{004}^2 \alpha_{301}-4 \alpha_{004} \alpha_{103} \alpha_{202}+\alpha_{103}^3)}, \,\, 
$$
and  $\,g_3 \,=\, u_3/v_3\,$ with
$$ \begin{matrix} u_3 & = &  64 \alpha_{004}^4 \alpha_{310}
+16 \alpha_{004}^2 \alpha_{013} \alpha_{301}
-16 \alpha_{004}^3 \alpha_{103} \alpha_{211} 
-8 \alpha_{004} \alpha_{013} \alpha_{103} \alpha_{202}   \\ & & 
+\,4 \alpha_{004}^2 \alpha_{103}^2 \alpha_{112} 
 + 2\alpha_{103}^3\alpha_{013}
-\alpha_{004} \alpha_{013} \alpha_{103}^3 ,\\
 v_3  & = &  
 2\alpha_{004} (8 \alpha_{004}^2 \alpha_{301}-4 \alpha_{004} \alpha_{103} \alpha_{202}+\alpha_{103}^3). 
 \qquad \qquad \qquad \qquad
 \end{matrix} $$

\smallskip

One can derive similar formulas for
the transformation to normal form when $d \geq 5$.
The denominator in the expressions for $g$ is the polynomial
of degree $d$ in $\alpha$ shown in (\ref{eq:genericity}).
\hfill $ \diamond$ \end{example}

Our task is to solve ${\rm discr}_z(\hat A) = B$, for
a fixed binary form $B$. This equation is understood projectively,
meaning that we seek  $\hat A$  in $\PP^{\binom{d+2}{2}-1}$
such that  ${\rm discr}_z(\hat A) $ vanishes at all zeros of $B$ in $\PP^1$.
By Theorem \ref{thm:normalform}, we may assume that
$\hat A$ lies in the subspace $L_d$.
Our system has extraneous solutions, namely
ternary forms $\hat A$ whose discriminant vanishes identically.
They must be removed when solving our recovery problem.
We now identify them geometrically.

\begin{proposition} \label{prop:baselocus}
The base locus of the discriminant map (\ref{eq:map1}) has two irreducible
components.
These have codimension $3$ and $2d-1$ respectively in $\,\PP^{\binom{d+2}{2}-1}$.
The former consists of all curves that are singular at $\,p = (0:0:1)$,
and the latter is the locus of non-reduced~curves.
\end{proposition}

\begin{proof}
The binary form ${\rm discr}_z(A)$ vanishes identically if and only if
the univariate polynomial function $z \mapsto A(u,v,z)$ has a double zero $\hat z$ for all $u,v \in \CC$.
If $p$ is a singular point of the curve $V(A)$ then $\hat z=0$ is always such a double zero.
If $A$ has a factor of multiplicity $\geq 2$ then so does the univariate
polynomial $z \mapsto A(u,v,z)$, and the discriminant vanishes.
Up to closure, we may assume that  this factor is a linear form,
so there are $\binom{d}{2}-1 + 2$ degrees of freedom. This shows that the family
of nonreduced curves $A$ has
codimension $2d-1 = (\binom{d+2}{2}-1) - (\binom{d}{2}+1)$.
The two scenarios define two distinct irreducible subvarieties of $\PP^{\binom{d+2}{2}-1}$.
For  $A$ outside their union, the binary form ${\rm discr}_z(A)$ is not identically zero.
\end{proof}

We now present our solution to the
recovery problem for cubic curves. 
Let $B$ be a binary sextic with six distinct zeros in $\PP^1$.
We are looking for a ternary cubic in the normal form
$$
A \,\,=\,\,
\alpha_{300} x^3 + 
\alpha_{201} x^2  z + 
\alpha_{111} x y z + 
\alpha_{102} x z^2 + 
\alpha_{030} y^3  + 
\alpha_{021} y^2 z + 
y z^2 + z^3.
$$
Here we assume $p=(0:0:1) \not\in V(A)$, so that
 $\alpha_{012} = \alpha_{003} = 1$.
We saw this in Theorem~\ref{thm:normalform}.
 The remaining six coefficients $\alpha_{ijk}$
 are unknowns.  The discriminant has degree three in these:
 $$ \! {\rm discr}_z(A) \! = \!
 (4 \alpha_{201}^3+27 \alpha_{300}^2) x^6 +(12 \alpha_{111} \alpha_{201}^2-18 \alpha_{201} \alpha_{300}) x^5 y
 + \cdots + 
 (4 \alpha_{021}^3-\alpha_{021}^2- \cdots +4  \alpha_{030})y^6.
$$
This expression is supposed to vanish at each of the six zeros of $B$.
This gives a system of six inhomogeneous cubic equations in the
six unknowns $\alpha_{ijk}$. In order to remove the extraneous solutions
described in  Proposition \ref{prop:baselocus},
we further require that the leading coefficient of the discriminant is nonzero.
We can write our system of cubic constraints in the $\alpha_{ijk}$ as follows:
\begin{equation}
\label{eq:system3}
 \begin{matrix} \quad {\rm rank} \begin{bmatrix} 
4 \alpha_{201}^3{+}27 \alpha_{300}^2 & 12 \alpha_{111} \alpha_{201}^2{-}18 \alpha_{201} \alpha_{300} & 
\cdots &  4 \alpha_{021}^3{-}\alpha_{021}^2- \cdots +4  \alpha_{030} \\
\beta_{60} & \beta_{51} & \cdots & \beta_{06} \end{bmatrix} \,\leq\, 1 \smallskip \\
{\rm and}\quad 4 \alpha_{201}^3+27 \alpha_{300}^2 \not= 0.  \qquad \qquad \qquad
 \qquad \qquad \qquad \qquad \qquad \qquad  \qquad \qquad \qquad 
\end{matrix} 
\end{equation}
This polynomial system exactly encodes the recovery of plane cubics
from six branch points.

\begin{corollary}\label{cor:deg3}
For general $\beta_{ij} $, the
system (\ref{eq:system3}) has  $\mathfrak{h}_3 = 40$ distinct solutions $\alpha \in \CC^6$.
\end{corollary}

\begin{proof}
The study of cubic curves tangent to a pencil of six lines
goes back to Cayley \cite{Cayley}. The formula
$\mathfrak{h}_3 = 40$ was found by Clebsch \cite{ClebschShort, ClebschLong}.
We shall discuss his remarkable work in Section~\ref{sec4}.
A modern proof for $\mathfrak{h}_3 = 40$ was given by
Kleiman and Speiser in \cite[Corollary~8.5]{KS}.

We here present the argument given in Ongaro's thesis \cite{Ongaro}.
By \cite[Proposition 5.2.2]{Ongaro},
every covering of $\PP^1$ by a plane cubic curve
is a shift in the group law of that elliptic curve followed by a linear projection
from a point in $\PP^2$. This implies that the classical Hurwitz
number, which counts such coverings, coincides with the plane
Hurwitz number $\mathfrak{h}_3$. The former is the number of
six-tuples $\tau = (\tau_1,\tau_2,\tau_3,\tau_4,\tau_5,\tau_6)$
of permutations of $\{1,2,3\}$, not all equal, whose
product is the identity, up to conjugation.
We can choose $\tau_1,\ldots,\tau_5$ in $3^5= 243$ distinct ways.
Three of these are disallowed, so there are $240$ choices.
The symmetric group $\mathbb{S}_3$ acts by conjugation on the
tuples $\tau$, and all orbits have size six. The number of classes
of allowed six-tuples is thus  $240/6 = 40$. This is our Hurwitz number $\mathfrak{h}_3$.
Now, the assertion follows from
Theorem~\ref{thm:normalform}, which ensures that
the solutions of (\ref{eq:system3}) are representatives.
\end{proof}

We next turn to another normal form, shown in (\ref{eq:othernf}),
 which has desirable geometric properties.
Let $A$ be a   ternary form  (\ref{eq:intro_f}) with
 $a_{00\,d} \not= 0$. We define a group element 
$g \in \mathcal{G}$ by
$$
g_0 = 1 \,, \,\,
 g_1 = -\frac{a_{10\,d-1}}{d \cdot a_{00d}} \, , \,\,
 g_2 = -\frac{a_{01\,d-1}}{d \cdot a_{00d}} \,, \,\,
  g_3 = 1.
$$
The coefficients of $xz^{d-1}$ and $yz^{d-1}$ in $gA$ are zero.
 Thus, after this transformation, we have
 \begin{equation}
\label{eq:othernf}
 A \,\,= \,\, z^d \,+\, A_2(x,y)\cdot z^{d-2} \,+\, A_{3}(x,y)\cdot z^{d-3} \,+ \,\cdots \,+\,
 A_{d-1}(x,y) \cdot z \,+ \, A_{d}(x,y) . 
\end{equation}
Here $A_i(x,y)$ is an arbitrary binary form of degree $i$. Its $i+1$ coefficients are unknowns.
The group~$\mathcal{G}$ still acts by rescaling $x,y$ simultaneously with
arbitrary non-zero scalars
 $\lambda \in \mathbb{C}^*$. 

 We next illustrate the utility of (\ref{eq:othernf}) by computing
 the planar Hurwitz number for $d{=}4$.
 Consider a general ternary quartic $A$.
We record
its $12$ branch points by fixing the discriminant $B = {\rm discr}_z(A)$.
Let $\hat A \in L_4$ be an unknown quartic in the normal form specified in
Theorem \ref{thm:normalform}, so $\hat A$ has $13$ terms,
$11$ of the form $\alpha_{ijk} x^i y^j z^k$ plus $y z^3$ and $z^4$. Our
task is to solve the following system of $12$ polynomial equations of degree five in the
$11$ unknowns $\alpha_{ijk}$:
\begin{equation}
\label{eq:system4}
\hbox{
Find all quartics $\hat A$ such that ${\rm discr}_z(\hat A)$ is a non-zero multiple of 
the binary form $B$.
}
\end{equation}

The number of solutions of this system was found by Vakil \cite{Ravi} with geometric methods.

\begin{theorem} \label{thm:120}
Let $B = \sum_{i+j=12} \beta_{ij} x^i y^j $ be the discriminant with respect to $z$ of a general ternary quartic $A$.
Then the polynomial
system (\ref{eq:system4}) has 
 $\mathfrak{h}_4 = 120$ distinct  solutions $\alpha \in \CC^{11}$.
\end{theorem}

The hypothesis ensures that $B$ is a point on Vakil's degree $3762$ hypersurface
$\mathcal{V}_4$ in $\PP^{12}$. This is a necessary and sufficient
condition for the system (\ref{eq:system4}) to have any solution at all.

\begin{corollary}
If we prescribe $11$ general branch points on the line $\PP^1$ then the number
of complex quartics $A$ such that ${\rm discr}_z( A)$ 
vanishes at these points is equal to 
$120 \cdot 3762 = 451440$.
\end{corollary}

\begin{proof}
Consider the space $\PP^{12}$ of binary forms of degree $12$.
Vanishing at $11$ general points defines a line in $\PP^{12}$.
That line meets the hypersurface $\mathcal{V}_4$ in $3762$ points.
By Theorem \ref{thm:120},
each of these points in $\mathcal{V}_4 \subset \PP^{12}$ has
precisely $120$ preimages $A$ in
$\PP^{14}$ under the map
(\ref{eq:map1}).
\end{proof}

\begin{remark} \label{rmk:extrafactor}
It was claimed in
\cite[equation (5.14)]{Ongaro} and \cite[page 608]{OS}
that $\mathfrak{h}_3$ is equal to
$120 \cdot (3^{10}-1)/2 =  3542880$.
That claim is not correct.
The factor $ (3^{10}-1)/2$ is not needed.
\end{remark}

\begin{proof}[Proof of Theorem \ref{thm:120}]
We work with the normal form (\ref{eq:othernf}).
Up to the  $\mathcal{G}$-action,  the  triples $(A_2,A_3,A_4)$  are parametrized by
   the $11$-dimensional weighted projective space   $ \mathbb{P}(2^3,3^4,4^5)$.
Following Vakil \cite{Ravi}, we consider a second weighted projective space of  dimension $11$, namely 
$\, \mathbb{P}(3^5, 2^7)$.
The weighted projective space $\mathbb{P}(3^5,2^7)$ parametrizes
pairs $(U_2,U_3)$ where
$U_i = U_i(x,y)$ is a binary form of degree $2i$, up to a
common rescaling of $x,y$ by some
 $\lambda \in \mathbb{C}^*$.

We define a rational map 
between our two weighted projective spaces as follows:
\begin{equation}
\label{eq:mapnu}
 \begin{matrix} \nu \,:\, \mathbb{P}(2^3,3^4,4^5)\, \dashrightarrow \,\mathbb{P}(3^5,2^7) 
\, , \,\, (A_2,A_3,A_4) \,\mapsto \, (U_2,U_3),  \qquad \qquad \smallskip \\
\qquad  {\rm where}  \quad 
 U_2 \,=\, -4A_4-\frac{1}{3}A_2^2 \quad {\rm and} \quad
 U_3 \,=\,  A_3^2-\frac{8}{3}A_2A_4 + \frac{2}{27}A_2^3. 
 \end{matrix} 
 \end{equation}
 We compose this with the following map into the space
 $\PP^{12} $ of binary forms of degree $12$:
 \begin{equation}
\label{eq:mapmu}
\mu \,:\,\mathbb{P}(3^5,2^7)  \, \dashrightarrow \, \PP^{12} \, , \,\,
(U_2,U_3) \, \mapsto \, 4\cdot U_2^3+27\cdot U_3^2.
\end{equation}
The raison d'\^{e}tre for the maps
(\ref{eq:mapnu}) and (\ref{eq:mapmu}) is that they represent
the formula of the discriminant ${\rm discr}_z(A)$ of
the special quartic in (\ref{eq:othernf}). Thus,
modulo the action of $\mathcal{G}$, we have $$ \pi  \,\,= \,\,\mu \,\circ\, \nu , $$ where $\pi: \PP^{14} \rightarrow \PP^{12}$ is
the branch locus map in (\ref{eq:map1}). One checks this by a direct computation.

Vakil proves in \cite[Proposition 3.1]{Ravi} that the map $\nu$ is dominant and its degree equals $120$. We also 
verified this statement independently via a numerical calculation in affine coordinates using \texttt{HomotopyContinuation.jl} \cite{BT},
and we certified its correctness using the method in \cite{BRT}.
This implies that the image of the map $\mu$ equals
the hypersurface $\mathcal{V}_4$. In particular,
$\mathcal{V}_4$ is the locus of all binary forms of degree
$12$ that are sums of the cube of a quartic and the square of a sextic.
Vakil proves in \cite[Theorem 6.1]{Ravi}  that the map 
$\mu$ is birational onto its image $\mathcal{V}_4$.
We verified this  statement by a Gr\"obner basis calculation.
This result
 implies that both $\nu$ and $\pi$ are maps of degree $120$, as desired.
\end{proof}

\begin{remark}
We also verified that $\mathcal{V}_4$ has  degree $3762$, namely by solving $12$ 
random affine-linear equations on the parametrization
(\ref{eq:mapmu}). The common Newton polytope of the
resulting polynomials has normalized volume $31104$.
This is the number of paths tracked by the polyhedral homotopy
in \texttt{HomotopyContinuation.jl}. We found $22572 = 3762 \times 6$ complex solutions.
The factor $6$ arises because $U_2$ and $U_3$ can be multiplied by roots of unity.
\end{remark}

\begin{algo} \label{algo:recovery4}
We implemented a numerical recovery method based
on the argument used to prove Theorem \ref{thm:120}.
The \underbar{input} is a pair $(U_2,U_3)$ as above.
The \underbar{output} consists of the
$120$ solutions in the subspace $L_4 \simeq \PP^{11}$ seen in (\ref{eq:Ld}).
We find these by solving the equations
\begin{equation}
\label{eq:raviU}
A_1 A_3-4 A_0 A_4- \frac{1}{3} A_2^2\, = \,U_2 \quad {\rm and} \quad
A_1^2 A_4 + A_0 A_3^2 - \frac{8}{3} A_0A_2A_4 -\frac{1}{3} A_1A_2A_3+\frac{2}{27}A_2^3\, =\, U_3.
 \end{equation}
 By \cite[Equation (5)]{Ravi}, these
   represent the discriminant 
for quartics $A =\sum_{i=0}^4 A_i z^{4-i}$.
To be precise, (\ref{eq:raviU}) is a system of
$12= 5+7$ equations in the $12 $ unknown coefficients of $A \in L_4$.
These have $120$ complex solutions, found easily with
\texttt{HomotopyContinuation.jl} \cite{BT}.
\end{algo}

\section{Hurwitz Combinatorics}
\label{sec3}

The enumeration of Riemann surfaces satisfying fixed ramification
was initiated by Hurwitz in his 1891 article \cite{Hurwitz}.
Hurwitz numbers are a widely studied subject, 
seen as central to combinatorial algebraic geometry.
For basics see \cite{CJM, CavalieriMiles, GMR, IZ, Ongaro}
and the references therein.

This paper concerns a general projection
$V(A)\rightarrow \mathbb{P}^1$ of a smooth plane curve of degree $d$ and genus 
$g=\binom{d-1}{2}$. In Section \ref{sec2} we studied the
inverse problem of recovering $A$ from the $d(d-1)$ simple
 branch points.
We now relate the
plane Hurwitz numbers $\mathfrak{h}_d$ to the
   Hurwitz numbers $H_d$ that count abstract covers.
   To be precise, $H_d$ is the number of degree $d$ covers $f$ of $\mathbb{P}^1$ by a genus $\binom{d-1}{2}$ curve $C$ having $d(d-1)$ fixed simple branch points. Each cover $f:C\rightarrow \mathbb{P}^1$ is weighted by $\frac{1}{|\Aut(f)|}$.
   Following \cite{CavalieriMiles},
the number $H_d$  can be found by counting monodromy representations, i.e.\ homomorphisms from the fundamental group of the target minus the branch points to the symmetric group over the fiber of the base point.

\begin{lemma}[Hurwitz \cite{Hurwitz}] \label{lem:abstract_hurwitz_number}
The Hurwitz number $H_d$ equals $1/d!$ times the number of tuples of transpositions $\tau = (\tau_1,\tau_2,\ldots,\tau_{d\cdot (d-1)})$ in 
the symmetric group $\mathbb{S}_d$  satisfying 
$$\tau_{d\cdot (d-1)}\circ\dots \circ \tau_2 \circ\tau_1 = \mathrm{id},$$ where the subgroup generated by the $\tau_i$ acts transitively on the set $\{1,2,\dots,d\}$. 
\end{lemma}

\begin{proposition}\label{prop:abstract_plane_numbers_relation}
For $d \geq 3$, the plane Hurwitz number is less than or equal to the 
classical Hurwitz number that counts abstract covers.
In symbols, we have $\,\mathfrak{h}_d \,\leq \,H_d$.
\end{proposition}

The restriction $d \geq 3$ is needed because of the weighted count, with automorphisms. For $d=2$, we have $H_2= 1/2$ because of the existence of a non-trivial automorphism
for maps $\PP^1 \rightarrow \PP^1$.
For higher $d$, the covers coming from projections of plane curves do not have automorphisms, so we can count them without this weight. This 
establishes Proposition \ref{prop:abstract_plane_numbers_relation}.

The two cases of primary interest in this paper are
$d=3$ and $d=4$. From the proofs of
Corollary \ref{cor:deg3} and  Theorem  \ref{thm:120},
we infer that the two cases exhibit rather different behaviors.

\begin{corollary} \label{cor:7528620}
For linear projections of cubic curves and quartic curves in $\PP^2$, we have
$$ \qquad \qquad \mathfrak{h}_3 \, = \, H_3 \, = \, 40  \qquad {\rm and} \qquad
\mathfrak{h}_4 \, = \, 120 \, \, < \,\,H_4 \,= \, 7528620.
$$
\end{corollary}

The count in Lemma \ref{lem:abstract_hurwitz_number} can be realized by combinatorial objects 
known as {\em monodromy graphs}. These occur in different guises
in the literature. We here use the version
that is defined formally in \cite[Definition 3.1]{GMR}.
These represent abstract covers in the tropical setting of 
balanced metric graphs.
We next list all
monodromy graphs for $d=3$.

\begin{example}[Forty monodromy graphs]
For $d=3$,  Lemma \ref{lem:abstract_hurwitz_number}
yields $H_3 = 40$ six-tuples
$\tau = (\tau_1,\tau_2,\ldots,\tau_6)$ of permutations of $\{1,2,3\}$,
up to the conjugation action by $\mathbb{S}_3$.
In Table~\ref{table:40covers} we list representatives for 
these $40$ orbits (see also \cite[Table 1]{Ongaro2}). Each tuple $\tau$ determines a
monodromy graph as in \cite[Lemma 4.2]{CJM} and \cite[Section 3.3]{GMR}.
Reading from the left to right, the diagram represents the
cycle decompositions of the permutations $\tau_i \circ \cdots \circ \tau_1$
for $i=1,\ldots,6$. For instance, for the first type $\mathcal{A}_1$,
we start at ${\rm id} = (1)(2)(3)$, then pass to $(12)(3)$, next to $(123)$, then to $(12)(3)$, etc.
On the right end, we are back at ${\rm id} = (1)(2)(3)$.

\begin{longtable}[H]{| c | c | c | c | c | c |c|}
\hline
$\!\!$ \textbf{Type}$\!$ &\textbf{Real?}$\!$ &  \textbf{Six-Tuple} $\tau$ & \textbf{Monodromy Graph} & \!\textbf{Clebsch}\! & 
$\!\mathbb{P}^3(\mathbb{F}_3)\!$
\\ \hline \hline
\makecell{$\mathcal{A}_1$ \\ $\mathcal{A}_2$  } & \makecell{\checkmark $ (12)$
\\ \checkmark $ (12)$ }  
&\makecell{ $(12)(13)(13)(13)(13)(12)$ \\ 
$ (12)(13)(13)(23)(23)(12)$}& \Lazypic{5cm}{
\includegraphics{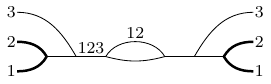}
 }
 & \makecell{ $ 123 $ \\ $ 1a $} 
 & \makecell{ $0010$ \\ $0100$}
  \\ \hline

\makecell{$\mathcal{A}_3$ \\$\mathcal{A}_4$ \\ $\mathcal{A}_{11}$ \\$\mathcal{A}_{12}$  }&\makecell{ \xmark \\ \xmark \\ \xmark \\ \xmark} & \makecell{ $(12)(13)(13)(13)(23)(13)$\\ $(12)(13)(13)(13)(12)(23)$ \\$(12)(13)(13)(23)(12)(13)$\\$(12)(13)(13)(23)(13)(23)$} & \Lazypic{5cm}{\includegraphics{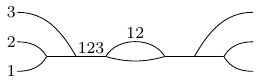}
 }  & \makecell{ $ 348 $ \\ $357$ \\ $7b$ \\ $4c$ } & \makecell{$ 1022 $ \\ $1012$ \\$1102$ \\ $1201$} \\ \hline

\makecell{$\mathcal{A}_5$ \\ $\mathcal{A}_6$\\ $\mathcal{A}_7$ \\$\mathcal{A}_{13}$ \\$\mathcal{A}_{14}$\\ $\mathcal{A}_{15}$}& \makecell{\xmark \\ \xmark \\ \xmark\\  \xmark \\ \xmark\\\xmark}& \makecell{ $(12)(13)(23)(23)(13)(12)$\\ $(12)(13)(23)(23)(23)(13)$\\ $(12)(13)(23)(23)(12)(23)$\\ $(12)(13)(23)(12)(23)(12)$\\ $(12)(13)(23)(12)(12)(13)$ \\$(12)(13)(23)(12)(13)(23)$ }& \Lazypic{5cm}{\includegraphics{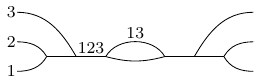}
 }& \makecell{ $456$ \\ $267$ \\ $ 168 $ \\ $1b$ \\ $7c$ \\ $4a$ } &\makecell{$1020$ \\ $1011$ \\ $0012$ \\ $1100$ \\ $1201$ \\ $0101$}\\ \hline

\makecell{$\mathcal{A}_8$ \\ $\mathcal{A}_9$ \\ $\mathcal{A}_{10}$\\$\mathcal{A}_{16}$ \\ $\mathcal{A}_{17}$ \\ $\mathcal{A}_{18}$ }&\makecell{ \xmark \\ \xmark \\ \xmark\\\xmark \\ \xmark \\ \xmark}&\makecell{ $(12)(13)(12)(12)(13)(12)$\\  $(12)(13)(12)(12)(23)(13)$ \\  $ (12)(13)(12)(12)(12)(23)$\\$(12)(13)(12)(13)(23)(12)$ \\$(12)(13)(12)(13)(12)(13)$\\$(12)(13)(12)(13)(13)(23)$  }& \Lazypic{5cm}{\includegraphics{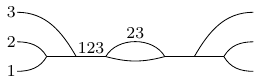}} & \makecell{ $789$ \\ $ 159 $ \\ $249$ \\ $1c$ \\ $7a$ \\ $4b$ } & \makecell{$1010$ \\ $0010$ \\ $1021$ \\ $1200$ \\ $0102$ \\ $1101$} \\ \hline

\makecell{$\mathcal{B}_1$ \\$\mathcal{B}_2$ } & 
\makecell{\checkmark (id) \\ \checkmark (id) }
 & \makecell{$(12)(12)(13)(13)(12)(12)$ \\  $(12)(12)(13)(13)(23)(23)$ } &  \Lazypic{5cm}{\includegraphics{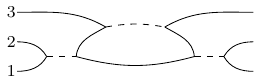}
}& \makecell{ base \\ $147$ }
& \makecell{ $1000$ \\ $0001 $} \\  \hline \hline

\makecell{$\mathcal{C}^{\ell}_1$ \\ $\mathcal{C}^{\ell}_2$ \\ $\mathcal{C}^{\ell}_3$ }& \makecell{\checkmark $(12)$ \\\xmark \\\xmark  }  & \makecell{$(12)(12)(12)(13)(13)(12)$ \\$(12)(12)(12)(13)(23)(13)$  \\  $(12)(12)(12)(13)(12)(23)$ } & \Lazypic{5cm}{ \includegraphics{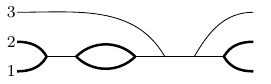}
}  & \makecell{$2a$ \\ $8b$ \\ $5c$} &  \makecell{$0110$ \\ $1112$  \\ $1222$} \\
\hline 

\makecell{$\mathcal{C}^{r}_1$ \\$\mathcal{C}^{r}_2$ \\ $\mathcal{C}^{r}_3$ }&  \makecell{\checkmark $(12)$ \\ \xmark \\ \xmark}  & \makecell{$(12)(13)(13)(12)(12)(12)$ \\$(12)(13)(23)(13)(13)(13)$ \\  $(12)(13)(12)(23)(23)(23)$ }& \Lazypic{5cm}{\includegraphics{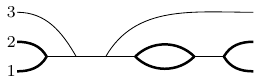}
 } & \makecell{$3a$ \\ $6b$ \\ $9c$} & 
 \makecell{$0120$ \\  $1121$ \\ $1211$} \\
\hline

\makecell{$\mathcal{D}^{\ell}_1$} &  \makecell{\checkmark  (id)  }  & $(12)(12)(12)(12)(13)(13)$&  \Lazypic{5cm}{\includegraphics{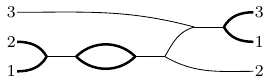}
 } & $369$ & $1002$\\ \hline\hline

\makecell{$\mathcal{D}^{r}_1$} & \makecell{\checkmark (id) } & $(12)(12)(13)(13)(13)(13)$ & \Lazypic{5cm}{\includegraphics{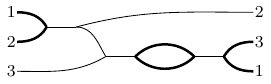} } & $258$ & $1001$ \\ \hline\hline

\makecell{$\mathcal{E}^{\ell}_1 $ \\ $\mathcal{E}^{\ell}_3 $ \\ $\mathcal{E}^{\ell}_5 $} &\makecell{ \xmark \\ \xmark \\ \xmark} & \makecell{$(12)(12)(13)(23)(13)(12)$\\$(12)(12)(13)(23)(23)(13)$ \\$ (12)(12)(13)(23)(12)(23)$ } & \Lazypic{5cm}{\includegraphics{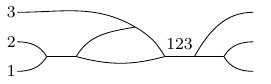}
 } & \makecell{$2b$ \\ $8c$ \\ $5a$} & \makecell{$1110$ \\ $1221$ \\ $0111$  }\\ \hline

\makecell{$\mathcal{E}^{\ell}_2 $ \\  $\mathcal{E}^{\ell}_4 $\\ $\mathcal{E}^{\ell}_6 $ }& \makecell{\xmark \\ \xmark \\ \xmark} & \makecell{$(12)(12)(13)(12)(23)(12)$\\$(12)(12)(13)(12)(12)(13)$\\$(12)(12)(13)(12)(13)(23)$ }& \Lazypic{5cm}{
\includegraphics{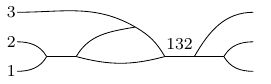}
 } & \makecell{$2c$ \\ $5b$ \\ $8a$ } & \makecell{$1220$ \\$1111$\\ $0112$} \\ \hline\hline

\makecell{$\mathcal{E}^{r}_1$\\ $\mathcal{E}^{r}_3$\\$\mathcal{E}^{r}_5$ }&\makecell{\xmark \\ \xmark \\ \xmark }&\makecell{$(12)(13)(23)(13)(12)(12)$\\$ (12)(13)(13)(12)(13)(13)$  \\$(12)(13)(13)(12)(23)(23) $  }&  \Lazypic{5cm}{ 
\includegraphics{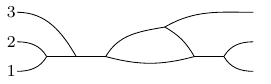}
} & \makecell{$3c$ \\ $6c$ \\ $9b$} & \makecell{$1210$ \\$1212$  \\ $1122$ } \\ \hline

\makecell{$\mathcal{E}^{r}_2$ \\ $\mathcal{E}^{r}_4$  \\$\mathcal{E}^{r}_6$  }&\makecell{ \xmark \\ \xmark \\ \xmark}& \makecell{$(12)(13)(12)(23)(12)(12)$ \\$(12)(13)(12)(23)(13)(13)$ \\$(12)(13)(23)(13)(23)(23)$}&  \Lazypic{5cm}{ 
\includegraphics{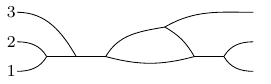}
} & \makecell{$3b$ \\ $6a$\\$9a$} & \makecell{$1120$ \\$0121$ \\ $0122$ } \\ \hline
\caption{The monodromy graphs for the $H_3=40$ coverings of $\PP^1$ by a genus one curve.
Eight of the $40$ coverings are real, and the certifying edge coloring is 
 shown in the graph.
 The two rightmost columns, labeled {\bf Clebsch} and $\,\PP^3(\mathbb{F}_3)$,
 will be explained in Section \ref{sec4}.
 }\label{table:40covers}
\end{longtable}

To identify real monodromy representations (see Lemma \ref{lem:real_abstract_hurwitz_numbers}), we give a coloring as in \cite[Definition 3.5]{GMR}. Using \cite[Lemma 3.5]{GMR} we find eight real covers among the $40$ complex covers. We use \cite[Lemma 2.3]{GMR} to associate the real covers to their monodromy representations. 
  
  We divide the $40$ classes into five types, $\mathcal{A}$ to $\mathcal{E}$, depending on the combinatorial type of the graph. Types $\mathcal{A}$ and $\mathcal{B}$ are symmetric under reflection of the ends, $\mathcal{C}$, 
  $\mathcal{D}$ and $\mathcal{E}$ are not. 
  An upper index $\ell$ indicates that the cycle of the graph is on the left side of the graph,
  while $r$ indicates that it is on the right side. The number of classes of each type is the multiplicity 
  in \cite[Lemma 4.2]{CJM}
and  \cite[Table 1]{Ongaro2}.
 Each class starts with the real types, if there are any, and proceeds lexicographically in $\tau$.
In the table, the edges of the monodromy graphs are labeled by the cycle they represent. If the edge is unlabeled,
then the corresponding cycle is either clear from context or varies through all possible cycles in $\mathbb{S}_3$ of appropriate length.
\hfill $ \diamond$ \end{example}

We now turn to  branched covers that are real.
In the abstract setting of Hurwitz numbers $H_d$, this has been
studied in \cite{Cadoret, GMR, IZ}.  A cover $f : C \rightarrow \PP^1$ is 
called {\em real} if the Riemann surface $C$ has an involution which is compatible with
complex conjugation on the Riemann sphere  $\PP^1$.
The branch points in $\PP^1$ can be real
or pairs of complex conjugate points.
We let $H^{\real}_d(r)$ be the weighted count of degree $d$ real covers $f$ of $\mathbb{P}^1$ by a genus $\binom{d-1}{2}$ curve $C$ having $d(d-1)$ fixed simple branch points, of which $r$ are real. As before, each cover $f:C\rightarrow \mathbb{P}^1$ is weighted by $\frac{1}{|\Aut(f)|}$.
The following result appears in \cite[Section 3.3]{Cadoret}.

\begin{lemma} \label{lem:real_abstract_hurwitz_numbers}
The real Hurwitz number $H^\real_d(r)$ equals $1/d!$ times the number of tuples $\tau$
as in Lemma \ref{lem:abstract_hurwitz_number} for which there exists an involution $\sigma \in \mathbb{S}_3$ such that $$\sigma\circ \tau_i\circ\dots\circ\tau_1\circ\sigma =  (\tau_1\circ\dots\circ\tau_i)^{-1}$$ for $i=1,\dots,r-1$ and $\sigma\circ\tau_{r+i}\circ\sigma=\tau_{r'+1-i}$ for $i = 1,\dots,r'$, where $r$ is the number of real branch points and $r'$ the number of pairs of complex conjugate branch points.
\end{lemma}

Geometrically, this means that, for a pair of complex conjugate points $q_1,q_2$, under complex conjugation the arc $\gamma_1$ around $q_1$ maps to $-\gamma_2$, where $\gamma_2$ is the arc around $q_2$.  Our next result says that the real Hurwitz number for
$d=3$ does not depend on $r$ and $r' =6-2r$.

\begin{proposition}\label{prop:real_abstract_hn_degree_3}
We have $H^{\real}_3(r)=8$ for $r=6,4,2,0$. 
\end{proposition}

\begin{proof}
We prove this by investigating all monodromy representations in Table~\ref{table:40covers}. Using explicit computations, we identify all six-tuples $\tau$ that satisfy the conditions in 
Lemma~\ref{lem:real_abstract_hurwitz_numbers}.
For a cover with $6$ real branch points, we obtain $8$ real monodromy representations, of types $\mathcal{A}_1, \mathcal{A}_2, \mathcal{B}_1 ,\mathcal{B}_2, \mathcal{C}^l_1, \mathcal{C}^r_1,\mathcal{D}^l_1$ and $ \mathcal{D}^r_1$, listed in Table \ref{table:40covers} with coloring. 
For a cover with $4$ real branch points and a pair of complex conjugate branch points, 
we again obtain $8$ real monodromy representations. These are the types $\mathcal{A}_3 , \mathcal{A}_{12}, \mathcal{B}_1 ,\mathcal{B}_2, \mathcal{C}^l_2, \mathcal{C}^r_1,\mathcal{D}^l_1$ and $ \mathcal{D}^r_1$. 
For two real branch points and two complex conjugate pairs, we again obtain $8$ real monodromy representations, namely of types  $\mathcal{A}_{9}, \mathcal{A}_{12}, \mathcal{B}_1 ,\mathcal{B}_2,   \mathcal{D}^l_1, \mathcal{D}^r_1, \mathcal{E}^{\ell}_3 $ and $\mathcal{E}^{r}_1$.
Finally, for three pairs of complex conjugate branch points, we find the $8$ types  $\mathcal{A}_{5}, \mathcal{A}_{17}, \mathcal{B}_1 ,\mathcal{B}_2 ,\mathcal{D}^l_1, \mathcal{D}^r_1, \mathcal{E}^{\ell}_3 $ and $\mathcal{E}^{r}_5$.
\end{proof}

The situation is more interesting for $d=4$, where we obtained the following result:

\begin{theorem} \label{thm:realcount4}
The real Hurwitz numbers for degree $4$ coverings of $\PP^1$ by genus $3$ curves are 
$$
\begin{matrix}
        H^{\real}_4(12)= 20590 , & 
        H^{\real}_4(10)= 15630 , &
        H^{\real}_4(8)= 11110 , &
        H^{\real}_4(6)= 7814 , 
\\        &
        H^{\real}_4(4)= 5654 , &
        H^{\real}_4(2) = 4070 , \,&
        H^{\real}_4(0)= 4350.
\end{matrix}        
$$
\end{theorem}

\begin{proof}
This is found by a direct computation using \textsc{Oscar} \cite{Oscar}.
We start by constructing a list of all monodromy representations of degree $4$ and genus $3$. As monodromy representations occur in equivalence classes, we construct only one canonical representative for each class. This is the element of the equivalence class that is minimal with respect to the lexicographic ordering.
The resulting list of $7528620$ monodromy representations was computed in about $6.5$ hours.
In other words, we embarked on a table just like Table~\ref{table:40covers},
but its number of rows is now $7528620$ instead of $40$. Those are the two numbers seen in
Corollary~\ref{cor:7528620}.

We next applied Cadoret's criterion in \cite[Section 3.3, formula $(\star)$]{Cadoret} to our big table.
This criterion was stated in Lemma \ref{lem:real_abstract_hurwitz_numbers}.
We start with our $ 7528620$ tuples $\tau$, computed as just described,
and mentioned in Lemma \ref{lem:abstract_hurwitz_number}.
According to Cadoret's criterion, we must check for each $12$-tuple $\tau$  whether there exists
an involution $\sigma$ that satisfies certain equations in the symmetric group $\mathbb{S}_4$.
These depend on the number $r$ of real branch points. Note that $ r = \{ 0,2,4,\ldots, 12\}$.
   For $r = 2, 4, \ldots, 12$, the only possible involutions $\sigma$ are $id$, $(12)$, $(34)$ and $(12)(34)$, by the structure of the canonical representative computed for the list.
   For $r = 0$, all involutions in $\mathbb{S}_4$ can appear.
      For each involution $\sigma$ and each value of $r$, it
took between $5$ and $30$ minutes to scan our big table, and to determine how many
$12$-tuples $\tau$ satisfy Cadoret's criterion for the pair $(r,\sigma)$.
For each $r$, we collected the number of tuples $\tau$ for which the
answer was affirmative. This gave the numbers stated in
Theorem \ref{thm:realcount4}.
\end{proof}

We next relate this Hurwitz combinatorics  to the polynomial systems in Section~\ref{sec2}.
Recall that we seek orbits of the group
$\mathcal{G}$ acting on $\PP^{\binom{d+2}{2}-1}$. An orbit is called
{\em real} if it has the form $\mathcal{G} A$ where $A$ is a ternary form with
real coefficients. Since $\mathcal{G}$ is defined over $\mathbb{R}$,
an orbit is real if and only if its unique intersection point with
the linear space $L_d$ in Theorem~\ref{thm:normalform}
is real. Thus, identifying the real orbits  among those with prescribed branch points
 is equivalent to deciding how many of the 
 $\mathfrak{h}_d$ complex solutions in our exact formulations (\ref{eq:system3}) and (\ref{eq:system4})
  are~real.

Suppose that the given binary form
 $B \in \mathcal{V}_d$ has real coefficients, and let $r$
  denote the number of real zeros of $B$. 
 In addition, there are $d(d-1)-2r$ pairs of complex conjugate zeros.
 It turns out that for $d=3$ the number of real solutions is independent of the number~$r$.
Our census of real plane Hurwitz numbers for quartics will be presented in Section \ref{sec5}.

\begin{corollary} \label{cor:from40to8}
The real plane Hurwitz number for cubics equals eight.
To be precise, the system  (\ref{eq:system3}) 
always has $8$ real solutions, 
provided the given parameters $\beta_{ij}$ are real and generic.
\end{corollary}

\begin{proof}
This is derived from
Corollary \ref{cor:7528620}
and Proposition~\ref{prop:real_abstract_hn_degree_3}. Namely,
we use the fact that plane covers are in bijection with abstract covers.
Let $C \rightarrow \PP^1$ be a real cover by an elliptic curve~$C$.
The involution of $C$ that is referred to in the proof of Corollary \ref{cor:deg3}
 is real as well.
 Another proof, following Clebsch \cite{ClebschShort, ClebschLong},
appears in Section~\ref{sec4}.
\end{proof}

\begin{algo}
\label{alg:recovery3}
We implemented numerical recovery for cubics that matches
Table \ref{table:40covers}. The \underbar{input} is a binary sextic $B$ with real coefficients.
The \underbar{output} consists of $40$ cubics $A$ in $L_3$ along
with their labeling by $\mathcal{A}_1,\mathcal{A}_2,\ldots, \mathcal{E}_6^r$.
The cubics are found with \texttt{HomotopyContinuation.jl} by solving
(\ref{eq:system3}). We next fix loops $\gamma_1,\gamma_2,\ldots,\gamma_6$
around the six roots of $B$  that are
compatible with complex conjugation on the Riemann sphere $\PP^1$.
 If all six roots are real then we use  \cite[Construction 2.4]{GMR}.
For each cubic $A$, we track the three roots $z$ of $A(x,y,z)=0$
as $(x:y)$ cycles along $\gamma_i$. The resulting permutation of the
three roots is the transposition~$\tau_i$. This process maps $A$ to a
tuple $\tau$ in Table~\ref{table:40covers}. This is unique up to conjugacy by $\mathbb{S}_3$.
The $8$ real cubics $A$ are mapped to the $8$ real monodromy representations,
in the proof of Proposition~\ref{prop:real_abstract_hn_degree_3}.
\end{algo}

\section{Cubics: Solutions in Radicals}
\label{sec4}

The theme of this paper is the rational map (\ref{eq:map1}) that takes a
ternary form to its $z$-discriminant.
This map is finite-to-one onto its image $\mathcal{V}_d$,
assuming the domain $\PP^{\binom{d+2}{2}-1}$ is understood modulo the group $\mathcal{G}$.
Note that $\mathcal{V}_d$
is an irreducible variety of dimension $\binom{d+2}{2}-4$ in $\PP^{d(d-1)}$.
The general fiber of the map consists of $\mathfrak{h}_d$ complex points.
We are curious about the Galois group ${\rm Gal}_d$
associated with this covering. Here {\em Galois group} is defined as in \cite{HarrisGalois}.
Informally, ${\rm Gal}_d$ is the subgroup
of geometry-preserving permutations of the $\mathfrak{h}_d$ solutions.

\begin{theorem} \label{thm:25920}
The Galois group ${\rm Gal}_3$ for cubics is the simple group of order $25920$, namely
\begin{equation}
\label{eq:weylrole} {\rm Gal}_3 \,\, =   \,\,{\rm SU}_4(\mathbb{F}_2) \,\, =
\,\, {\rm PSp}_4(\mathbb{F}_3) \,\, = \,\, W(E_6)/\! \pm. 
\end{equation}
This is the Weyl group of type $E_6$ modulo its center, here
realized as $4 \times 4$ matrix groups over the finite fields
$\mathbb{F}_2$ and $\mathbb{F}_3$.
 The action of ${\rm Gal}_3$ on the $40$
monodromy graphs in  Table \ref{table:40covers}
agrees with that of the symplectic group
on the $40$ points in the projective space $\PP^3$ over $\mathbb{F}_3$.
\end{theorem}

This is a modern interpretation of 
  Clebsch's work \cite{ClebschShort, ClebschLong} on
binary sextics. We~first learned about the role of the Weyl group in (\ref{eq:weylrole})
through  Elkies' unpublished manuscript~\cite{elkies}.

\begin{remark}
The last two columns of Table~\ref{table:40covers}
identify the $40$ monodromy representations with
$\PP^3(\mathbb{F}_3)$ and with Clebsch's $40$ cubics in \cite{ClebschLong}.
The bijection we give respects the maps
${\rm Gal}_3 \simeq {\rm PSp}_4(\mathbb{F}_3) \hookrightarrow \mathbb{S}_{40}$.
But it is  far from unique. The same holds for
Cayley's table in \cite{CayleyAlgo}.
\end{remark}

\begin{proof}[Proof of Theorem~\ref{thm:25920}]
We consider cubics $  \,A =  z^3 + A_2(x,y)  z + A_3(x,y)$.
This is the normal form in (\ref{eq:othernf}).
The discriminant equals $\,{\rm discr}_z(A) = 4 A_2^3 + 27 A_3^2$.
Thus our task is as follows: \smallskip \\
{\em Given a binary sextic $B $, compute all pairs of binary forms
$(A_2,A_3)$  such that $4 A_2^3 + 27 A_3^2 = B$}.

\smallskip
This system has $40$ solutions, modulo the scaling of $A_2$ and $A_3$ by roots of unity.
Setting $U = \sqrt[3]{4} \cdot A_2 $ and
$V = \sqrt{-27} \cdot A_3$, we must solve 
 the following problem:
{\em Given a binary sextic $B $, compute all decompositions
into a binary quadric $U$ and a binary cubic $V$:}
\begin{equation}
\label{eq:BUVequation}
 B \,\, = \,\, U^3-V^2 . 
 \end{equation}
This is precisely the problem addressed by 
Clebsch in \cite{ClebschShort, ClebschLong}.
By considering the change of his labeling upon altering the base solution, he implicitly determined the Galois group as a subgroup of $\mathbb{S}_{40}$. The identification
of this group with $W(E_6)$ modulo its center
appears in a number of sources, including \cite{Hunt, Todd}.
These sources show that ${\rm Gal}_3$ is also the Galois group
of the  $27$ lines on the cubic surface.
Todd \cite{Todd} refers to permutations of the
 $40$ Jacobian planes, and Hunt \cite[Table 4.1]{Hunt}
 points to the $40$ triples of trihedral pairs.
The connection to cubic surfaces goes back to Jordan in 1870, 
and it was known to  Clebsch.

As a subgroup of the symmetric group $\mathbb{S}_{40}$, our Galois group is generated by
five permutations $\Gamma_1,\ldots,\Gamma_5$. These correspond to consecutive transpositions
$(\gamma_i \gamma_{i+1})$
of the six loops
$\gamma_1,\gamma_2,\ldots,\gamma_6$ in
 Algorithm \ref{alg:recovery3}. 
Each generator is a product of nine $3$-cycles in $\mathbb{S}_{40}$. 

Here are the formulas for
$\Gamma_1,\ldots,\Gamma_5$
as permutations of the $40$ rows in Table~\ref{table:40covers}:
\begin{small}
$$  \!
\Gamma_1 = 
 (\mathcal{A}_{10}  \mathcal{A}_6  \mathcal{A}_1\!)
 (\mathcal{A}_8  \mathcal{A}_7  \mathcal{A}_3\!)
 (\mathcal{A}_9  \mathcal{A}_5  \mathcal{A}_4\!)
 (\mathcal{A}_{17}  \mathcal{A}_{13}  \mathcal{A}_{12}\!)
 (\mathcal{A}_{18}  \mathcal{A}_{14}  \mathcal{A}_2\!)
 (\mathcal{A}_{16}  \mathcal{A}_{15}  \mathcal{A}_{11}\!)
 (\mathcal{E}^r_2  \mathcal{E}^r_6  \mathcal{E}^r_3)
 (\mathcal{E}^r_4  \mathcal{E}^r_1  \mathcal{E}^r_5)
 (\mathcal{C}^r_3  \mathcal{C}^r_2  \mathcal{C}^r_1)
$$
$$
\Gamma_2 \,= 
 (\mathcal{E}^\ell_4  \mathcal{A}_{14}  \mathcal{A}_{10})
 (\mathcal{E}^\ell_6  \mathcal{A}_{15}  \mathcal{A}_9)
 (\mathcal{E}^\ell_2  \mathcal{A}_{13}  \mathcal{A}_8)
 (\mathcal{B}_1  \,\mathcal{E}^r_1  \mathcal{E}^r_2)
 (\mathcal{D}^r_1  \,\mathcal{C}^r_2  \mathcal{C}^r_3)
 (\mathcal{B}_2  \,\mathcal{E}^r_6 \mathcal{E}^r_4)
 (\mathcal{E}^\ell_5  \mathcal{A}_7  \mathcal{A}_{17})
 (\mathcal{E}^\ell_1  \mathcal{A}_5  \mathcal{A}_{16})
 (\mathcal{E}^\ell_3  \mathcal{A}_6  \mathcal{A}_{18})
$$
$$
\Gamma_3 \,=\,
 (\mathcal{C}^\ell_3 \, \mathcal{E}^\ell_5 \, \mathcal{E}^\ell_4)
 (\mathcal{C}^\ell_1 \, \mathcal{E}^\ell_1\,  \mathcal{E}^\ell_2)
 (\mathcal{C}^\ell_2  \,\mathcal{E}^\ell_3  \,\mathcal{E}^\ell_6)
 (\mathcal{A}_{17}  \mathcal{A}_{11}  \mathcal{A}_{14})
 (\mathcal{A}_{18}  \mathcal{A}_{12}  \mathcal{A}_{15})
 (\mathcal{A}_{16}  \mathcal{A}_2  \mathcal{A}_{13})
 (\mathcal{E}^r_2  \mathcal{E}^r_1  \mathcal{C}^r_1)
 (\mathcal{E}^r_4  \mathcal{C}^r_2  \mathcal{E}^r_3)
 (\mathcal{C}^r_3  \mathcal{E}^r_6  \mathcal{E}^r_5)
$$
$$
\Gamma_4 =
 (\mathcal{D}^\ell_1  \,\mathcal{C}^\ell_2 \, \mathcal{C}^\ell_3)
 (\mathcal{E}^\ell_6 \, \mathcal{B}_2 \, \mathcal{E}^\ell_5)
 (\mathcal{E}^\ell_2  \,\mathcal{E}^\ell_1  \,\mathcal{B}_1)
 (\mathcal{A}_8  \mathcal{A}_{16}  \mathcal{E}^r_2)
 (\mathcal{A}_9  \mathcal{E}^r_4  \mathcal{A}_{17})
 (\mathcal{E}^r_3  \mathcal{A}_3  \mathcal{A}_{11})
 (\mathcal{E}^r_5  \mathcal{A}_{12}  \mathcal{A}_4)
 (\mathcal{A}_{15}  \mathcal{E}^r_6  \mathcal{A}_7)
 (\mathcal{A}_{13}  \mathcal{A}_5  \mathcal{E}^r_1)
$$
$$
\Gamma_5 =
 (\mathcal{C}^\ell_3  \mathcal{C}^\ell_2  \mathcal{C}^\ell_1)
 (\mathcal{E}^\ell_4  \mathcal{E}^\ell_6  \mathcal{E}^\ell_2)
 (\mathcal{E}^\ell_5  \mathcal{E}^\ell_3  \mathcal{E}^\ell_1)
 (\mathcal{A}_{10}  \mathcal{A}_9  \mathcal{A}_8\!)
 (\mathcal{A}_{17}  \mathcal{A}_{18}  \mathcal{A}_{16}\!)
 (\mathcal{A}_4  \mathcal{A}_3  \mathcal{A}_1\!)
 (\mathcal{A}_{11}  \mathcal{A}_{12}  \mathcal{A}_2\!)
 (\mathcal{A}_{14}  \mathcal{A}_{15}  \mathcal{A}_{13}\!)
 (\mathcal{A}_7  \mathcal{A}_6  \mathcal{A}_5)
$$
\end{small}
A compatible bijection with the labels of Clebsch
\cite[Section 9]{ClebschLong} is given in the second-to-last column in 
Table~\ref{table:40covers}. The last column gives a 
bijection with the $40$ points in the projective space
$\PP^3$ over the three-element field $\mathbb{F}_3$.
This bijection is compatible with the action of the matrix group
$ {\rm PSp}_4(\mathbb{F}_3) $. Here, the five generators above are mapped to 
matrices of order $3$:
 \begin{small} $$ \Gamma_1 = \begin{bmatrix} 
1 \! & \! 1 \! & \! 2 \! & \! 0 \\
0 \! & \! 1 \! & \! 0 \! & \! 0 \\
0 \! & \! 0 \! & \! 1 \! & \! 0 \\
0 \! & \! 1 \! & \! 2 \! & \! 1 
\end{bmatrix}\!, \, \Gamma_2 = \begin{bmatrix} 
1 \! & \! 0 \! & \! 0 \! & \! 0 \\
2 \! & \! 1 \! & \! 0 \! & \! 2 \\
1 \! & \! 0 \! & \! 1 \! & \! 1 \\
0 \! & \! 0 \! & \! 0 \! & \! 1 
\end{bmatrix}\!, \,\Gamma_3 = \begin{bmatrix} 
1 \! & \! 1 \! & \! 0 \! & \! 0 \\
0 \! & \! 1 \! & \! 0 \! & \! 0 \\
0 \! & \! 0 \! & \! 1 \! & \! 0 \\
0 \! & \! 0 \! & \! 0 \! & \! 1 
\end{bmatrix}\!, \,\Gamma_4 = \begin{bmatrix} 
1 \! & \! 0 \! & \! 0 \! & \! 0 \\
2 \! & \! 1 \! & \! 0 \! & \! 1 \\
2 \! & \! 0 \! & \! 1 \! & \! 1 \\
0 \! & \! 0 \! & \! 0 \! & \! 1 
\end{bmatrix}\!, \,\Gamma_5 = \begin{bmatrix} 
1 \! & \! 1 \! & \! 1 \! & \! 0 \\
0 \! & \! 1 \! & \! 0 \! & \! 0 \\
0 \! & \! 0 \! & \! 1 \! & \! 0 \\
0 \! & \! 2 \! & \! 2 \! & \! 1 
\end{bmatrix}\!.
$$

\end{small} 
These are symplectic matrices with entries in $\mathbb{F}_3$, modulo 
scaling by $(\mathbb{F}_3)^* = \{\pm 1\} = \{1,2\}$.
A computation using \textsc{GAP} \cite{GAP} verifies that these groups are indeed isomorphic. In 
the notation of the atlas of simple groups, our Galois group (\ref{eq:weylrole})
 is the group $O_5(3)$.
 \end{proof}

\begin{remark}
The fundamental group of the configuration space of
six points in the Riemann sphere $\PP^1$ is the braid group $B_6$, which therefore maps
onto the finite group ${\rm Gal}_3$. One checks that the permutations and
matrices listed above satisfy the braid relations that define $B_6$:
$$
\Gamma_i \,\Gamma_{i+1} \,\Gamma_i \,=\,
\Gamma_{i+1} \,\Gamma_i\, \Gamma_{i+1} \quad \hbox{for $i=1,2,3,4$} \quad {\rm and} \quad
\Gamma_i \,\Gamma_j = \Gamma_j \,\Gamma_i \quad \hbox{for $| i - j | \geq 2$.}
$$
\end{remark}

We conclude from
 Theorem \ref{thm:25920}
that our recovery problem for cubics
is solvable in~radicals.

\begin{corollary}
Starting from a given ternary cubic $A$, the $39$ other solutions $(U,V)$ to the equation (\ref{eq:BUVequation})
can be written in radicals in the coefficients of the binary forms $A_2$ and $A_3$.
\end{corollary}

\begin{proof}
Viewing ${\rm Gal}_3$ as a group of permutations on the $40$ solutions,
we consider the stabilizer subgroup of the one given solution. That stabilizer 
has order $25920/40 = 3 \cdot 216$,
and this is the Galois group of the other $39$ solutions.
It contains the Hesse group ${\rm ASL}_2(\mathbb{F}_3)$ as a 
 normal subgroup of index $3$. Recall that ${\rm ASL}_2(\mathbb{F}_3)$
 is solvable, and has order $216$. It is the Galois group
 of the nine inflection points of a plane cubic \cite[Section~II.2]{HarrisGalois}.
 Therefore the stabilizer is solvable, and hence $(U,V)$ is expressable in radicals over $(A_2,A_3)$.
 \end{proof}

Clebsch explains how to write the
$39$ solutions in radicals in the coefficients of $(A_2,A_3)$.
We now give a brief description of his algorithm, 
which reveals the inflection points of a~cubic.

\begin{algo} \label{alg:clebsch}
Our \underbar{input} is a pair $(A_2,A_3)$ that represents a cubic $A$ as in (\ref{eq:othernf}).
We set $\tilde U = \sqrt[3]{4} \cdot A_2 $ and
$\tilde V = \sqrt{-27} \cdot A_3$. Our \underbar{output} is the list of all pairs
$(U,V)$ that satisfy
\begin{equation}
\label{eq:UVsystem}  U^3 - {\tilde U}^3 \,\, = \,\,
      V^2 - {\tilde V}^2. 
\end{equation}      
Real solutions of (\ref{eq:system3})
correspond to pairs $(U,V)$ such that $U$ and $iV$ are real, where $i = \sqrt{-1}$.
To solve (\ref{eq:UVsystem}), consider the cubic 
$\mathcal{C} = \{ (x:y:z) \in \PP^2: z^3 - 3\tilde U z + 2 \tilde V = 0\} $.
The nine inflection lines of $\mathcal{C}$ are defined by
the linear forms $\xi = \alpha x + \beta y $
such that  $\xi^3 - 3\tilde U \xi + 2 \tilde V $ is the cube of a  linear form $\eta = \gamma x + \delta y$.
We write the coefficients $\alpha,\beta$ of such $\xi$ in radicals. Here, the Galois group is
${\rm ASL}_2(\mathbb{F}_3)$.
We next compute $(\gamma,\delta)$ from
$\eta^3 =  \xi^3 - 3\tilde U \xi + 2 \tilde V$. 
For each of the pairs $(\alpha,\beta)$ above,
this system has  three solutions $ (\gamma,\delta) \in \CC^2$.
This leads to $27$ vectors $(\alpha,\beta,\gamma,\delta)$, all expressed in radicals.
For each of these we set $U = \tilde U - \xi^2 - \xi \eta - \eta^2$.~Then
$$ U^3 - ({\tilde U}^3 - {\tilde V}^2) \, = \, - \frac{3}{4}
\bigl( \eta^3 + 2 \eta^2 \xi + 2 \eta \xi^2 + \xi^3 - 2 \eta \tilde U - \xi \tilde U\bigr)^2 . $$
The square root of the right hand side is a binary cubic $V$, and the pair $(U,V)$ satisfies 
(\ref{eq:UVsystem}). In this manner we construct $27$ solutions to (\ref{eq:system3}),
three for each inflection point of the curve $\mathcal{C}$. Three of the $9$ inflection points are real,
and each of them yields one real solution to (\ref{eq:system3}).

We finally compute the $12 = 39-27$ remaining solutions, of which four are real.
For this, we label the inflection points of $\mathcal{C}$ by $1,2,\ldots,9$ such that the following triples
are collinear:
\begin{equation}
\label{eq:twelve}
 \underbar{123},456,789,\quad \underbar{147},\underbar{258},\underbar{369},\quad 159,267,348,
 \quad 168,249,357 . 
\end{equation} 
If $1,2,3$ are real
and $\{4,7\}$, $\{5,8\}$, $\{6,9\}$ are complex conjugates, then precisely the four underlined
lines are real. Our labeling agrees with that used by Clebsch in
\cite[\S 9, page~50]{ClebschLong}.

We now execute the formulas in \cite[\S 11, \S 12]{ClebschLong}.
For each of the $12$ lines in (\ref{eq:twelve}), we compute
two solutions $(U,V)$ of  (\ref{eq:UVsystem}).
These are expressed rationally in the data $(\alpha,\beta)$ found above.
Each new solution $(U,V)$ arises twice from each triple of lines in (\ref{eq:twelve}),
so we get $12$ in total.
\end{algo}

\section{Quartics:  Del Pezzo Strikes Again}
\label{sec5}

This section features threads from algebraic geometry that inform
our recovery problem for quartics. We begin
with the extension of
Corollary \ref{cor:from40to8} from $d=3$ to $d=4$.
In other words, we determine the
plane counterparts to the real Hurwitz numbers
in Theorem \ref{thm:realcount4}.

\begin{theorem} \label{thm:realcount4planar}
Consider the polynomial system (\ref{eq:system4}) for
recovering quartics, where the parameters $\beta_{ij}$ are generic  reals.
The number of real solutions
equals $8$, $16$, $24$, $32$, $64$, or~$120$.
\end{theorem}

\begin{proof}
We use the connection to del Pezzo surfaces of degree one and tritangents of space sextic
curves, described by Vakil in~\cite{Ravi}. We recall some parts of his construction,
and we examine what changes when passing from the complex numbers to the real numbers. 

The open set $\mathcal{B}$ of $\PP(3^5,2^7)$ is birational to the hypersurface $\mathcal{Z}$ in $\PP^{12}$ which parametrizes possible branch data. Hence we obtain from a general binary form $B$ in $\mathcal{Z}$ an element $(U_2,U_3)$ of $\PP(3^5,2^7)$. Fix the space sextic $\mathcal{C}$ by $X_2^3 + U_2(X_0,X_1)X_2 + U_3(X_0,X_1)$ in the singular cone $\PP(1,1,2)$. Sending $(U_2,U_3)$ to $\mathcal{C}$ gives the isomorphism $\mathcal{B} \cong \mathcal{E}$ in \cite[Theorem 2.12]{Ravi}, where $\mathcal{E}$ is the space of genus $4$ curves
 whose canonical model 
lies on a quadric cone in $\PP^3$.

In the complex setting taking a double cover of $\PP(1,1,2)$ branched over $\mathcal{C}$ gives a del Pezzo surface of degree one, together with involution $\iota$, called the {\em Bertini involution}.
Over the real numbers we obtain two different del Pezzo surfaces, depending on the choice of a sign. We denote the two distinct real surfaces by $X$ and $X'$. Upon base changing to the complex numbers,
the surfaces $X_\CC$ and $X'_\CC$ become isomorphic. 
 
The possible topological types of real del Pezzo surfaces have been classified by several authors.
We here follow the work of Russo \cite{Russo}. Blowing down a $(-1)$-curve $L$ on $X_\CC$ gives a del Pezzo surface of degree $2$ which is a double cover of a unique real plane quartic. The Bertini involution $\iota$ of $X$ maps $L$ to a $(-1)$-curve giving the same quartic. The pair $\{L,\iota(L)\}$ produces a real quartic curve if and only if is invariant under the complex conjugation on $X_\CC$. In this case we call it a {\em real Bertini pair}.
The real plane Hurwitz number is the number of  real Bertini pairs on $X$. Mapping $L$ to a line in the cone $\PP(1,1,2)$ gives a tritangent to 
the space sextic $\mathcal{C}$ \cite[Example 4]{Russo}. In this way,
the real tritangents of $\mathcal{C}$ are the same as the real Bertini pairs. The former have been counted. We reproduce the table in \cite[Corollary 5.3]{Russo} in Table \ref{table:real_del_pezzo} and explain his notation for the types of del Pezzo surfaces. 

First, we have the del Pezzo surfaces of blowup type: The surface obtained by blowing up $2r$ real points and $4-r$ complex conjugate pairs of points is denoted by $\PP^2(2r,8-2r)$. With the Bertini involution one can define the real structure denoted by $\mathbb{B}_1$. Similarly, the Geiser involution on real del Pezzo surfaces of degree $2$ give rise to the real surface $\mathbb{G}_2$. Finally, the birational de Jonquieres involutions of the plane give rise to del Pezzo surfaces $\mathbb{D}_2$ and $\mathbb{D}_4$, where the index indicates the degree.  
\end{proof}
\begin{table}
\centering
\begin{tabular}{|p{2 cm}|p{2 cm}|p{2 cm}|}
\hline
$\mathfrak{h}_4^{real}(B)$ & $X$ & $X'$ \\ 
\hline 
$120$ & $\mathbb{B}_1$ & $\PP^2(8,0)$ \\
$64$ & $\mathbb{G}_2(1,0)$ & $\PP^2(6,2)$ \\
$32$ & $\mathbb{D}_2(1,0)$ & $\PP^2(4,4)$ \\
$16$ & $\mathbb{D}_4(1,2)$ & $\PP^2(2,6)$ \\
$8$ & $\PP^2(0,8)$ & $\PP^2(0,8)$ \\
$24$ & $\mathbb{D}_4(3,0)^0_3$ & $\mathbb{D}_4(3,0)^0_3$ \\
$24$ & $\mathbb{D}_4(3,0)^1_2$ & $\mathbb{D}_4(3,0)^1_2$ \\
\hline
\end{tabular}
\caption{$\mathfrak{h}_4^{real}(B)$ is the number of real Bertini pairs}
\label{table:real_del_pezzo}
\end{table}

\begin{remark} \label{rmk:consistent}
Algorithm \ref{algo:recovery4} is consistent with Theorem~\ref{thm:realcount4planar}.
Given binary forms $U_2,U_3$ with real coefficients, it  outputs
$\,8$, $16$, $24$, $32$, $64$ or $120$ real quartics in the subspace~$L_4$.
\end{remark}

\begin{corollary} \label{cor:rationalinstance}
From any general configuration $\mathcal{P} = \{P_1,P_2,\ldots,P_8\}
\subset \PP^2$ with coordinates in $\QQ$,
we obtain an instance of (\ref{eq:system4}) whose $120$ complex solutions  $A$
all have coefficients in $\QQ$.
\end{corollary}

\begin{proof}
We construct the $120$ quartics $A$
with rational arithmetic from the coordinates in $\mathcal{P}$. Fix $j$ and
let $w$ be a cubic that vanishes on $\mathcal{P} \backslash \{P_j\}$
but does not vanish at $P_j$. Consider the
$2$-$1$ map $\PP^2 \dashrightarrow \PP^2$ given by $(u:v:w)$.
The branch locus of this map is a quartic curve.

We claim that this is the quartic $Q$ that corresponds to the
exceptional fiber of the blow-up at $P_j$ in Vakil's construction. Indeed, let $X$ be the blowup of $\mathbb{P}^2$ at all the $9$ base
points of the pencil $(u:v)$ and let $C\subseteq X$ be a general curve in the pencil. The intersection of the quartic $Q$ with $C$ is given by \cite[Proposition 3.1]{Ravi} as the set of points $p \in C$ such that $\mathcal{O}_{C}(2p) \cong \mathcal{O}_{C}(E_0+E_1)$. 
Alternatively, these are exactly the branch points of the map $C\to \mathbb{P}^1$ induced by the linear system of $\mathcal{O}_{C}(E_0+E_1)$. Thus, we are done if we can prove that the restriction of the map $(u:v:w)\colon X \to \mathbb{P}^2$ is given exactly by the linear system above. However, the map of $(u:v:w)$ on $X$ corresponds to the complete linear system $L = 3H-E_2-\dots-E_8 = -K_X+E_0+E_1$ so that $\mathcal{O}_{C}(L) \cong \mathcal{O}_{C}(-K_X+E_0+E_1)$ and we need to show that $\mathcal{O}_C(-K_X)\cong \mathcal{O}_C$. But this follows from the adjunction formula, using the fact that $C$ is an elliptic curve moving in a pencil.

Thus, we can recover the quartic as this branch locus of our
$2$-$1$ map $\PP^2 \dashrightarrow \PP^2$.
 The ternary form $A$ that defines this branch locus can be computed from the ideals of $\mathcal{P}$ and $P_j$
using Gr\"obner bases, and this uses rational arithmetic over the ground field.
This yields eight of the $120$ desired quartics.
The remaining $112$ quartics are found by repeating this process
after some Cremona transformations  have been applied to 
the pair $(\PP^2,\mathcal{P})$. These transformations use only
rational arithmetic and they preserve
the del Pezzo surface, but they change the collection of eight
$(-1)$-curves that are being blown down. 
\end{proof}

The following two algorithms provide more details on the steps in the proof above.

\begin{algo} \label{alg:get8}
The \underbar{input} is a set $\mathcal{P}$ of eight rational points in $\PP^2$.
The \underbar{output} is eight quartics $A \in \QQ[x,y,z]$
that share the same $z$-discriminant. These are the quartics corresponding to the eight exceptional divisors $E_i$.
Write $(r:s:t)$ for the coordinates of the $\PP^2$ that contains $\mathcal{P}$.
We start by setting up
the following ideal in $\,\QQ[r,s,t,x,y,z]$:
\begin{equation}
\label{eq:jacobian}  I \,\,\,=\,\,\,
\bigl\langle \hbox{$ 2 \times 2$ minors of} \,
\begin{pmatrix} u & v & w \\ x & y & z \end{pmatrix}
\bigr\rangle \,\,+\,\,
\biggl\langle
{\rm det} \! \begin{pmatrix}
 \partial u / \partial r &  \partial u / \partial s &  \partial u / \partial t \\
  \partial v / \partial r &  \partial v / \partial s &  \partial v / \partial t \\
 \partial w / \partial r &  \partial w / \partial s &  \partial w / \partial t \\
\end{pmatrix} \biggr\rangle.
\end{equation}
Saturate $I$ with respect to the ideal of $\mathcal{P}\backslash \{P_j\}$.
After that step, eliminate the three variables $r,s,t$. The result is a
principal ideal in $\QQ[x,y,z]$ whose generator is the quartic $A$.
We perform the above two steps for $j=1,2,\ldots,8$, and we output the resulting eight quartics.
\end{algo}

Our second algorithm concerns the Cremona transformations mentioned in the proof.

\begin{algo} \label{alg:get120}
Our \underbar{input} is a set $\mathcal{P}$ of eight rational points in $\PP^2$.
The \underbar{output} is the remaining set of $112$ quartics $A \in \QQ[x,y,z]$
that also have same $z$-discriminant. Fix a basis of cubics $u,v$ passing through $\mathcal{P}$. Choose a triple $\{P_i,P_j,P_k\}$
and set up the associated quadratic Cremona transformation $\PP^2 \dashrightarrow \PP^2$.
This transforms $\mathcal{P}$ into a new configuration $\mathcal{P}'$ and it corresponds to an automorphism of the del Pezzo surface. Apply Algorithm \ref{alg:get8}
to $\mathcal{P}'$ and record the resulting eight quartics. In applying Algorithm \ref{alg:get8}, we need some care in choosing a basis $\{u',v'\}$ for the cubics vanishing at $\mathcal{P}'$: we should choose them in such a way that the elliptic pencil $(u':v')\colon \mathbb{P}^2 \dashrightarrow \mathbb{P}^1$ is the composition of the original pencil $(u:v)\colon \mathbb{P}^2 \dashrightarrow \mathbb{P}^1$ and the quadratic Cremona transformation. This ensures that the quartics share the same $z$-discriminant with the original eight quartics.  We now repeat this process until
$120$ inequivalent quartics have been found.
This is possible because the Cremona transformations 
act transitively on the set of $(-1)$ curves on the del Pezzo surface. 
\end{algo}

\begin{remark}
 If we use Algorithm \ref{alg:get120} to construct branched covers from real plane quartics 
 then we cannot obtain  $24$ real solutions. This is because  the associated surfaces are  not of blowup type, see Table \ref{table:real_del_pezzo}. Nonetheless we found instances with real plane Hurwitz  $24$ by constructing a sextic $\mathcal{C}$ which has this number of tritangents. It can be seen in Figure \ref{figure:24_sextic}. The equation of $\mathcal{C}$ gives us $(U_4,U_6)\in\mathcal{B}$ and with Algorithm \ref{algo:recovery4} we recover $24$ quartics.
\end{remark}

\begin{figure}[h]
    \centering
    \includegraphics[width=0.5\linewidth]{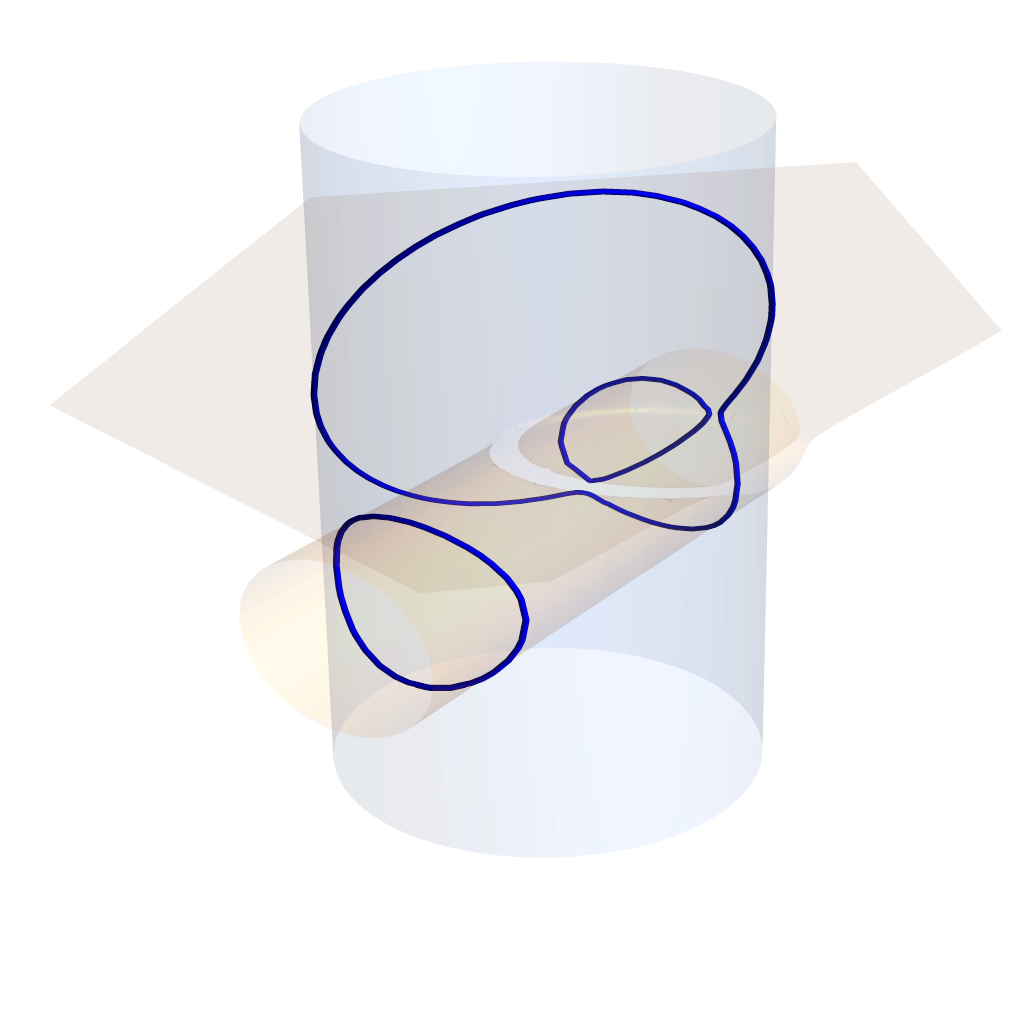}  \vspace{-0.5in}
    \caption{A space sextic with exactly $24$ real tritangents.}
    \label{figure:24_sextic}
   \end{figure}

Using Algorithm \ref{alg:get8}, we can find one quartic $A$ and hence
the binary form $B = {\rm discr}_z(A)$  using rational arithmetic.
However, there is an even more direct method which also shows that
$B$ does not depend on the choice of the point $P_j$. Namely,
$B$ encodes the $12$ singular curves in the pencil of
cubics through $\mathcal{P}$. Based on this observation, we derive our final~result:

\begin{theorem} \label{thm:rleqs}
Let $B $ be the binary form of degree $12$ computed from  a real configuration $\mathcal{P}$.
Suppose $2r$ of the $8$ points in $\, \mathcal{P}$ are real
and $2s$ of the $12$ zeros of $B$  are real.
Then $r \leq s$. Conversely, for all valid parameters $r \leq s$, some
configuration $\mathcal{P}$ realizes the pair $(r,s)$.
\end{theorem}

\begin{proof}
The $12$ points in $\mathbb{P}^1$ which are zeros of $B$ correspond to the singular fibers under the map $\mathbb{P}^2 \rightarrow \mathbb{P}^1: (x:y:z) \mapsto (u:v)$.
This  map is given by two cubics $u$,$v$ through~$\mathcal{P}$. Algebraically,
we seek points $(\alpha:\beta) \in \mathbb{P}^1$ such that the cubic  $ \beta \cdot u - \alpha \cdot v$  is singular. There are
$12$ such singular cubics, and the question is how many of those can be~real.
 The {\em Welschinger invariant}
 $W_r(\mathbb{P}^2,3)$ is a signed count of real singular cubics passing through $2r$ real points and $4-r$ pairs of complex conjugate points. One counts a singular cubic with a hyperbolic node with a positive sign and one with an elliptic node with a negative sign. By definition, $W_r(\mathbb{P}^2,3)\leq 2s$, the number of real solutions. Example 4.3 in \cite{Shustin} provides a tropical computation of the Welschinger invariants for cubics and yields $W_r(\mathbb{P}^2,3)=2r$. We conclude $2r\leq 2s$.
The last sentence in Theorem  \ref{thm:rleqs}
is proved with a direct  calculation. By sampling from rational configurations $\mathcal{P} =
\{P_1,\ldots,P_8\}$,
we found instances for all $(r,s)$ with $0 \leq r \leq s \leq 6$ and $r \leq 4$. These can be found at our
 {\tt MathRepo} website
\href{https://mathrepo.mis.mpg.de/BranchPoints/}{https://mathrepo.mis.mpg.de/BranchPoints}$\,$.
\end{proof}

\begin{remark}
If we take one real quartic $Q$ corresponding to our real configuration $\mathcal{P}$, 
then we can see which of the $2r$ real tangents belong to elliptic nodal cubics. To see this recall how to construct an elliptic fibration from $Q$. 
We blow up the point $(0:0:1)$ and then take a double cover branched over the quartic \cite[Section 2.1]{Ravi}. The nodal fibers are the preimages of the tangents of $Q$ passing through the base point. The ovals of $Q$ divide 
$\PP^2_\RR$
 into connected components. We mark the component containing $(0:0:1)$ with a $+$, and we
 proceed so that adjacent components have opposite signs. Let $L$ be one of the $2r$ real tangents and $q$ be its points of tangency with $Q$. If $L$ lies in a positive component in a neighborhood of $q$, then $L$ gives a hyperbolic node. Otherwise the preimage of $L$ in the fibration is an elliptic node. One proves this by writing down the equation of the double cover in a suitable affine neighborhood of $q$. The case of $10$ hyperbolic nodal curves can be seen in Figure \ref{figure:12_tangents}.
\end{remark}
\begin{figure}[h]
    \centering
    \includegraphics[width=0.5\linewidth]{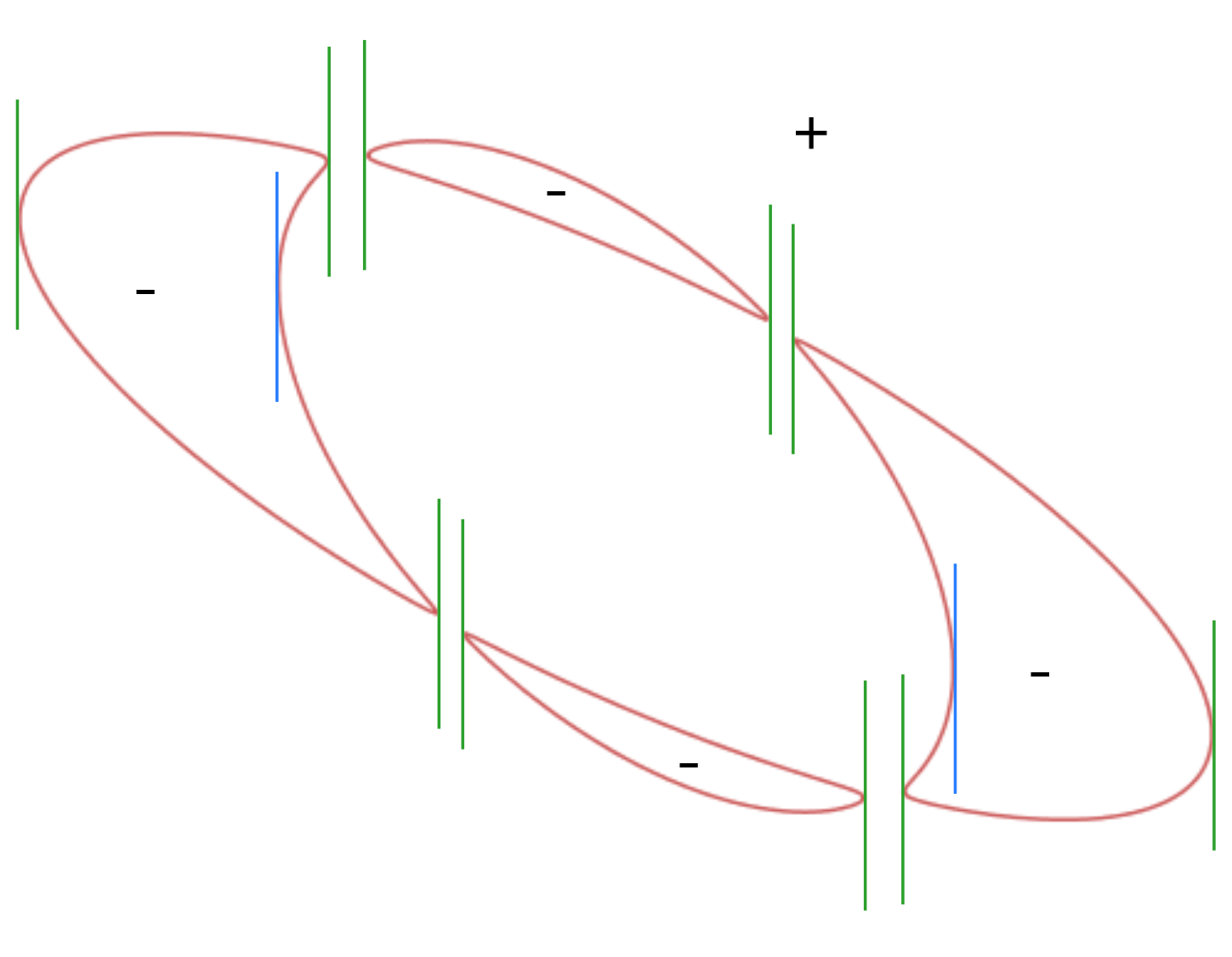}
    \caption{A real quartic with $12$ real tangents}
    \label{figure:12_tangents}
\end{figure}

\bigskip

\subsection*{Acknowledgements}
We thank Arthur Bik, Lou-Jean Cobigo, 
Yelena Mandelstham, Jared Ongaro, Boris Shapiro, Simon Telen, Rainer Sinn and Mario Kummer 
for helpful discussions. Furthermore we thank the anonymous referees for careful reading and useful comments.
Hannah Markwig, Victoria Schleis and Bernd Sturmfels were supported by the Deutsche Forschungsgemeinschaft (DFG, German Research Foundation), Project-ID 286237555, TRR 195. 
Javier Sendra--Arranz received the support of a fellowship from the ``la Caixa'' Foundation (ID 100010434).
  The fellowship code is LCF/BQ/EU21/11890110.

\bigskip
\bigskip 

\noindent
\footnotesize
{\bf Authors' addresses:}

\smallskip

\noindent Daniele Agostini,
Universit\"at T\"ubingen
\hfill {\tt daniele.agostini@uni-tuebingen.de}

\noindent Hannah Markwig,
Universit\"at T\"ubingen
\hfill {\tt hannah@math.uni-tuebingen.de}

\noindent Clemens Nollau,
Universit\"at T\"ubingen
\hfill {\tt clemens.nollau@uni-tuebingen.de}

\noindent Victoria Schleis,
Universit\"at T\"ubingen
\hfill {\tt victoria.schleis@student.uni-tuebingen.de}

\noindent Javier Sendra--Arranz,
MPI-MiS Leipzig
\hfill {\tt javier.sendra@mis.mpg.de}
             
\noindent Bernd Sturmfels,
MPI-MiS Leipzig and UC Berkeley
\hfill {\tt bernd@mis.mpg.de}

\end{document}